\documentclass[11pt]{article}

\usepackage{makeidx}  % allows for indexgeneration
\usepackage[english]{babel}
\usepackage{amssymb,amsmath,enumerate,a4}
\usepackage{latexsym}
\usepackage{graphicx}
\usepackage{epsfig}
\usepackage{pstricks}
\usepackage{pst-node}
\usepackage{colortab}
\usepackage{color}
\usepackage{framed}

\newcommand{\conp}{\tt coNP}
\newcommand{\np}{\tt NP}

\newcommand{\vRLF}{\tt vRLF}
\newcommand{\IPRI}{\tt IPRI}
\newcommand{\RCGR}{\tt RCGR}
\newcommand{\IF}{\tt IF}
\newcommand{\vol}{\mathrm{vol}}
\newcommand{\bb}{\mathbb}
\newcommand{\R}{\bb R}
\newcommand{\Q}{\bb Q}
\newcommand{\Z}{\bb Z}
\newcommand{\N}{\bb N}
\newcommand{\floor}[1]{\left\lfloor#1\right\rfloor}

\newcommand{\conv}{\operatorname{conv}}

\newcommand{\cone}{\operatorname{cone}}
\newcommand{\reccone}{\operatorname{rec}}
\newcommand{\intcone}{\operatorname{int{.}cone}}
\newcommand{\linhull}[1]{\langle #1 \rangle}
\newcommand{\affhull}{\operatorname{aff}}
\newcommand{\interior}{\operatorname{int}}
\newcommand{\bd}{\operatorname{bd}}
\newcommand{\aff}{\operatorname{aff}}
\newcommand{\dime}{\operatorname{dim}}
\newcommand{\sm}{\setminus}

\newcommand{\relint}{\operatorname{relint}}

\newtheorem{Prop}{Proposition}
\newtheorem{Obs}[Prop]{Observation}
\newtheorem{lem}[Prop]{Lemma}

\newtheorem{thm}[Prop]{Theorem}

\newtheorem{cl}{Claim}
\newtheorem{Rm}{Remark}

\newenvironment{proof}{\begin{trivlist} \item[] {\em Proof.}}
{\hspace*{\stretch{1}} $\Box$ \end{trivlist}}

\newenvironment{cpf}{\begin{trivlist} \item[] {\em Proof.}}
{\hspace*{\stretch{1}} $\diamond$ \end{trivlist}}

\title{Reverse Chv\'atal-Gomory rank\thanks{This work was supported by the {\em Progetto di Eccellenza 2008--2009} of the {\em Fondazione Cassa di Risparmio di Padova e Rovigo}. A previous version has appeared in the proceedings of IPCO 2013.}}

\author{Michele Conforti\thanks{Dipartimento di Matematica, Universit\`a degli Studi di Padova, Italy.}
\and Alberto Del Pia\thanks{Department of Industrial and Systems Engineering $\&$ Wisconsin Institute for Discovery, University of Wisconsin-Madison, United States, {\tt delpia@wisc.edu}.}
\and Marco Di Summa\footnotemark[2]
\and Yuri Faenza\thanks{DISOPT, Institut de math\'ematiques
d'analyse et applications, EPFL Switzerland, {\tt yuri.faenza@epfl.ch}. Supported by the German Research Foundation (DFG) as part of the priority program  “SPP 1307: Algorithm Engineering.}
\and Roland Grappe\thanks{Laboratoire d'Informatique, Universit\'e Paris-Nord, France, {\tt roland.grappe@lipn.univ-paris13.fr}.}}

\begin{document}

\maketitle

\begin{abstract}
We introduce the \emph{reverse Chv\'atal-Gomory rank} $r^*(P)$ of an integral polyhedron $P$, defined as the supremum of the Chv\'atal-Gomory ranks of all rational polyhedra whose integer hull is $P$. A well-known example in dimension two shows that there exist integral polytopes $P$ with $r^*(P)=+\infty$. We provide a geometric characterization of polyhedra with this property in every dimension, and investigate upper bounds on $r^*(P)$ when this value is finite. \\[4pt]
\end{abstract}

\section{Introduction}\label{sec:intro}

A polyhedron is \emph{integral} if it is the convex hull of its integer points. Given an integral polyhedron $P\subseteq\R^n$, a \emph{relaxation} of $P$ is a rational polyhedron $Q \subseteq \R^n$ such that $Q\cap\Z^n=P\cap\Z^n$. Note that if $Q$ is a relaxation of $P$, then $P=\conv(Q\cap\Z^n)$, i.e., $P$ is the {\em integer hull} of $Q$, where we denote the convex hull of a set $S$ by $\conv(S)$ (for the definition of convex hull and other standard preliminary notions not given in here, we refer the reader to textbooks, e.g.~\cite{Grun} and~\cite{sch}). An inequality $cx\le\floor{\delta}$ is a \emph{Chv\'atal--Gomory inequality} (\emph{CG inequality} for short) for a polyhedron $Q\subseteq \R^n$ if $c$ is an integer vector and $cx\le\delta$ is valid for $Q$. Note that $cx\le\floor\delta$ is a valid inequality for $Q\cap\Z^n$. The \emph{CG closure} $Q'$ of $Q$ is the set of points in $Q$ that satisfy all the CG inequalities for $Q$. If $Q$ is a rational polyhedron, then $Q'$ is again a rational polyhedron~\cite{sch80}. For $p\in\N$, the {\em $p$-th CG closure} $Q^{(p)}$ of $Q$ is defined iteratively as $Q^{(p)}=(Q^{(p-1)})'$, with $Q^{(0)}=Q$. If $Q$ is a rational polyhedron, then there exists some $p \in \N$ such that $Q^{(p)}=\conv(Q\cap\Z^n)$~\cite{sch80}. The minimum $p$ for which this occurs is called the \emph{CG rank} of $Q$ and is denoted by $r(Q)$.

Cutting plane procedures in general and CG inequalities in particular are of crucial importance to the integer programming community, because of their convergence properties (see e.g.~\cite{Gom,sch}) and relevance in practical applications (see e.g.~\cite{cplex}). Hence, a theoretical understanding of their features has been the goal of several papers from the literature. Many of them aimed at giving upper or lower bounds on the CG rank for some families of polyhedra. For instance, Bockmayr et al.~\cite{Bock} proved that the CG rank of a polytope $Q\subseteq[0,1]^n$ is at most $O(n^3\log n)$. The bound was later improved to $O(n^2\log n)$ by Eisenbrand and Schulz~\cite{EiSc}. Recently, Rothvo\ss{} and Sanit\`a~\cite{RothSan}, improving over earlier results of Eisenbrand and Schulz~\cite{EiSc} and Pokutta and Stauffer~\cite{PoSt}, showed that this bound is almost tight, as there are polytopes in the unit cube whose CG rank is at least $\Omega(n^2)$. An upper bound on the CG rank for polytopes contained in the cube $[0,\ell]^n$ for an arbitrary given $\ell$ was provided by Li~\cite{Li}. Recently, Averkov et al.~\cite{non-full-dim} studied the rate of convergence -- in terms of number of iterations of the CG closure operator -- of the affine hull of a rational polyhedron to the affine hull of its integer hull.

\medskip

{\bf Our contribution.} In this paper we investigate a question that is, in a sense, reverse to that of giving bounds on the CG rank for a fixed polyhedron $Q$. In fact, in most applications, even if we do not have a complete linear description of the integer hull $P$, we know many of its properties: for instance, the integer points of most polyhedra
stemming from combinatorial optimization problems have 0-1 coordinates. Hence, for a fixed \emph{integral} polyhedron $P$, we may want to know how ``bad'' a relaxation of $P$ can be in terms of its CG rank. More formally, we want to answer the following question: given an integral polyhedron $P$, what is the supremum of $r(Q)$ over all rational polyhedra $Q$ whose integer hull is $P$? We call this number the \emph{reverse CG rank} of $P$ and denote it by $r^*(P)$:
\[r^*(P)=\sup\{r(Q):\mbox{$Q$ is a relaxation of $P$}\}.\]
Note that $r^*(P) < +\infty$ if and only if there exists $p \in \N$ such that $r(Q) \le p$ for every relaxation $Q$ of $P$.
Our main result gives a geometric characterization of those integral polyhedra $P$ for which $r^*(P)=+\infty$. Recall that the \emph{recession cone} of a polyhedron $P$ is the set of vectors $v$ such that $x+ \alpha v \in P$ for each $x \in P$ and $\alpha \in \R_+$. Denoting by $\reccone(P)$ the recession cone of $P$, by $\linhull v$ the line generated by a non-zero vector $v$, and by $+$ the Minkowski sum of two sets, we prove the following:

\begin{thm}\label{th:main}
Let $P\subseteq\R^n$ be an integral polyhedron.
Then $r^*(P) = +\infty$ if and only if $P$ is non-empty and there exists $v \in \Z^n \sm \reccone (P)$ such that $P + \linhull v$ does not contain any integer point in its relative interior.
\end{thm}

\smallskip

Theorem \ref{th:main} can also be interpreted in terms of proof systems for integer programming. As the CG procedure always terminates in finite time, CG cuts provide one such proof system. One of the main aims of this research area is to understand what are the ``obstacles'' on the way to determination of the integer hull using those proof systems (see e.g.~\cite{BaCoMa,DelP}). In our context, this boils down to highlighting the features of a given integral polyhedron which make it difficult to be ``proved'' using CG cuts. Theorem \ref{th:main} gives an exact description of those obstacles, characterizing when the finiteness of the CG procedure is an intrinsic property of the integer hull rather than a property of one of its infinitely-many relaxations.

\smallskip

Let us illustrate Theorem~\ref{th:main} with an example in dimension two. Let $P=\conv\{(0,0),(0,1)\}$, and consider the family $\{Q_t\}_{t \in \N}$ of relaxations of $P$, where we define $Q_t=\conv\{(0,0),(0,1),(t,1/2)\}$. It is folklore that the CG rank of $Q_t$ increases linearly with $t$ (see Figure \ref{fig:2-dim}). This implies that $r^*(P)=+\infty$. Note that if one chooses $v=(1,0)$, then $P+\linhull v$ does not contain any integer point in its (relative) interior. A simple application of Theorem~\ref{th:main} shows that the previous example can be generalized to every dimension: any 0-1 polytope $P\subseteq \R^n$, $n\geq 2$, whose dimension is at least $1$, has infinite reverse CG rank, since there always exists a vector $v$ parallel to one of the axis such that $P+\linhull v$ does not contain any integer point in its relative interior. On the other hand, every integral polyhedron containing an integer point in its relative interior (e.g. one with full-dimensional recession cone) has finite reverse CG rank, as no vector $v$ satisfying the condition of Theorem~\ref{th:main} exists in this case. However, there are also integral polyhedra with finite reverse CG rank that do not contain integer points in their interior, such as $\conv\{(0,0),(2,0),(0,2)\}\subseteq\R^2$.

We then show that for a wide class of polyhedra with finite reverse CG rank, $r^*$ can be upper bounded by functions depending only on parameters such as the dimension of the space and the number of the integer points in the relative interior of $P$. Moreover, we give examples showing that $r^*$ of those polyhedra grows with those parameters.

Last, we investigate some algorithmic issues. In particular, we show that the problem ``does an integral polyhedron $P$ have finite reverse CG rank?'' can be decided in finite time, and in polynomial time if the dimension is fixed.

Results of this paper are proved combining classical tools from cutting plane theory (e.g. the lower bound on the CG rank of a polyhedron by Chv\'atal, Cook, and Hartmann~\cite{ChCoHa}, see Lemma \ref{lem:iterat-cc}) with geometric techniques that are not usually
applied to the theory of CG cuts, mostly from geometry of numbers (such as the characterization of maximal lattice-free convex sets~\cite{BasuDirichelet}, or Minkowski's Convex Body Theorem).

\begin{figure}
\center
% Generated with LaTeXDraw 2.0.8
% Wed May 30 17:50:26 CEST 2012
% \usepackage[usenames,dvipsnames]{pstricks}
% \usepackage{epsfig}
% \usepackage{pst-grad} % For gradients
% \usepackage{pst-plot} % For axes
\scalebox{.8} % Change this value to rescale the drawing.
{
\begin{pspicture}(0,-2.12)(9.72,2.12)
\definecolor{color738b}{rgb}{0.8,0.8,0.8}
\definecolor{color739b}{rgb}{0.6,0.6,0.6}
\definecolor{color740b}{rgb}{0.4,0.4,0.4}
\psline[linewidth=0.04cm,linestyle=dashed,dash=0.16cm 0.16cm,arrowsize=0.05291667cm 2.0,arrowlength=1.4,arrowinset=0.4]{<-}(1.0,2.1)(1.0,-2.1)
\psline[linewidth=0.04cm,linestyle=dashed,dash=0.16cm 0.16cm,arrowsize=0.05291667cm 2.0,arrowlength=1.4,arrowinset=0.4]{->}(0.0,-0.9)(9.7,-0.88)
\pspolygon[linewidth=0.04,fillstyle=solid,opacity=0.5,fillcolor=color738b](1.0,1.16)(1.0,-0.84)(6.9,0.08)
\pspolygon[linewidth=0.04,fillstyle=solid,opacity=0.5,fillcolor=color739b](0.98,1.12)(0.98,-0.88)(4.96,0.1)
\pspolygon[linewidth=0.04,fillstyle=solid,fillcolor=color740b](0.98,1.14)(0.98,-0.86)(2.92,0.04)
\psdots[dotsize=0.4](1.0,-0.9)
\psdots[dotsize=0.4](1.0,1.1)
\psdots[dotsize=0.4,dotangle=-90.0](3.0,-0.9)
\psdots[dotsize=0.4,dotangle=-90.0](5.0,-0.9)
\psdots[dotsize=0.4,dotangle=-90.0](7.0,-0.9)
\psdots[dotsize=0.4,dotangle=-90.0](9.0,-0.9)
\psdots[dotsize=0.4,dotangle=-90.0](3.0,1.1)
\psdots[dotsize=0.4,dotangle=-90.0](5.0,1.1)
\psdots[dotsize=0.4,dotangle=-90.0](7.0,1.1)
\psdots[dotsize=0.4,dotangle=-90.0](9.0,1.1)
\end{pspicture}
}
\caption{In increasingly lighter shades of grey, polytopes $Q_1$, $Q_2$, and $Q_3$.}\label{fig:2-dim}
\end{figure}
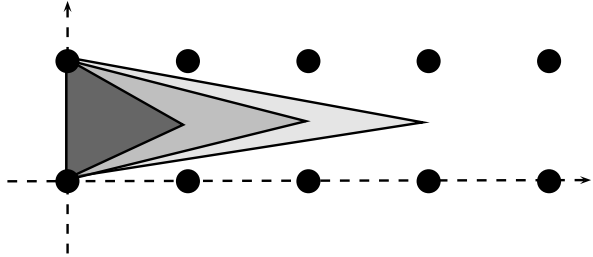

%We then focus on two families of polyhedra with finite reverse CG rank and investigate which parameters their reverse CG rank depends on (Theorem \ref{thr:bounds}).

%We then shift our attention to the stronger family of \emph{splits} cuts. If again we assume that the input polyhedron $Q$ is rational, the split operator share with the CG operator the polyhedrality of the closure and the finite rank (in fact, it is a generalization of the CG operator). Moreover, computational experiments (citazioni...) show that the split operator converges to the integer hull much faster than then CG operator. However, this practical evidence is not mirrored by theoretical results: with the notable exception of the polyhedra contained in the unit cube (cit: Balas), there is no theorem proving that split cuts are more effective than CG cuts. It then makes sense to investigate whether the concept of reverse rank provides such a theoretical evidence, that is, if there exist substantial classes of polyhedra whose reverse CG rank is notably bigger than the analogously defined \emph{reverse split rank}. In the last part of the paper, we initiate this study. We show that the reverse split rank of any integer polyhedron in the cartesian space is at most 2 (recall that Figure \ref{fig:2-dim} provides an example of an integer polygon $P$ with $r^*(P)=+\infty$). On the other hand, we show we cannot hope for a finite reverse split rank for all integer polyhedra, providing a polytope $P\subseteq \R^3$ with infinite reverse split rank.

\medskip

%Recall that a relaxation of an integral polyhedron is a {\em rational} polyhedron by definition. We remark that the rationality assumption is crucial in the statement of Theorem~\ref{th:main}. As an example, consider the polytope $P\subseteq \R^2$ consisting only of the origin. Any line $Q\subseteq \R^2$ passing through the origin and having irrational slope is an (irrational) polyhedron whose integer hull is $P$. One readily verifies that the CG closure of $Q$ is $Q$ itself, showing that in this case the CG closures of $Q$ do not converge to the integer hull $P$. However, no vector $v$ satisfying the conditions of Theorem~\ref{th:main} exists.

%After recalling and stating some auxiliary results in Section~\ref{sec:def}, we prove Theorem~\ref{th:main} (Section~\ref{sec:main}). We then consider two specific families of polyhedra having finite reverse rank and investigate the parameters on which the reverse rank depends on (Section~\ref{sec:cg-more}).

\smallskip

The paper is organized as follows. In Section \ref{sec:def}, we settle notation and definitions, and state some known and new auxiliary lemmas needed in the rest of the paper. In Section \ref{sec:main}, we prove the main result of the paper, that is the geometric characterization of integral polyhedra with infinite reverse CG rank (Theorem \ref{th:main}). In Section \ref{sec:cg-more}, we focus on two classes of polyhedra with finite reverse CG rank and investigate upper bounds on $r^*$ for those classes. Section \ref{sec:algo} is devote to algorithmic issues. We conclude with Section \ref{sec:extens}, where some extensions of the concept of reverse CG rank are examined.

\section{Definitions and tools}\label{sec:def}

Throughout the paper, $n$ will be a strictly positive integer denoting the dimension of
the ambient space. Given a set $S\subseteq\R^n$, we denote by $\intcone(S)$ the set of all linear combinations of vectors in $S$ using nonnegative integer multipliers. Given a closed, convex set $C \subseteq \R^n$, the {\em affine hull} of $C$, denoted $\aff(C)$, is the smallest affine subspace containing $C$. The {\em dimension} of $C$ is the dimension of $\aff(C)$. $C$ is {\em full-dimensional} if its dimension is $n$. We also denote by $\bd(C)$ the boundary of $C$, by $C_I$ the integer hull of $C$, by $\interior(C)$ the interior of $C$, by $\relint(C)$ the relative interior of $C$. We say that $C$ is \emph{lattice-free} if $\interior(C) \cap\Z^n= \emptyset$, and \emph{relatively lattice-free} if $\relint(C)\cap\Z^n=\emptyset$. Note that the relative interior of a single point in $\R^n$ is the point itself. Hence, if it is integer, then it is not relatively lattice-free. Also, note that if $C$ is not lattice-free, then it is full-dimensional. A \emph{convex body} is a closed, convex, bounded set with non-empty interior. A set $C$ is {\em centrally symmetric} with respect to a given point $x\in C$ (or centered at $x$) when, for every $y\in\R^n$, one has $x+y\in C$ if and only if $x-y\in C$.
%\footnote{It has been suggested e.g. in~\cite{nizi} to call those sets \emph{hollow} instead of \emph{lattice-free}, since the latter is used in the literature to denote a different family of convex sets. However, the integer programming community seem to agree on the name \emph{lattice-free}, hence we stick to this.}

By \emph{distance} between two points $x, y \in \R^n$ (resp. a point $x\in\R^n$ and a set $S\subseteq\R^n$) we mean the Euclidean distance, which we denote by $d(x,y)$ (resp. $d(x,S)$). We use the standard notation $\|\cdot\|$ for the Euclidean norm. For $r \in \Q_+$, $x \in \R^n$ and an affine subspace $H\subseteq\R^n$ of dimension $d$, the \emph{$d$-ball} (\emph{of radius $r$ lying on $H$ and centered at $x$}) is the set of points lying on $H$ whose distance from $x$ is at most $r$. When referring to the volume of a $d$-dimensional convex set $C$, denoted $\vol(C)$, we shall always mean its $d$-dimensional volume, that is, the Lebesgue measure with respect to the affine subspace $\aff(C)$ of the
Euclidean space $\R^n$.

\subsection*{Bounds on the CG rank}

We give here upper and lower bounds on the CG rank of polyhedra. The proof of the following two results can be found in~\cite{CCT} and~\cite{non-full-dim} respectively.

\begin{lem}\label{lem:PI=0}
Each rational polyhedron $Q\subseteq \R^n$ with $Q_I=\emptyset$ has CG rank at most $\varphi(n)$, where $\varphi$ is a function depending on $n$ only.
\end{lem}

\begin{lem}\label{lem:upper}
For every polyhedron $Q \subseteq \R^n$
and for every $a \in \Z^n$ and $\delta,\delta'\in\R$ (with $\delta'\ge\delta$) such that $ax \le
\delta$ is valid for $Q_I$ and $ax \le \delta'$ is valid for $Q$, the inequality $ax
\le \delta$ is valid for $Q^{(p+1)}$, where $p=(\lfloor \delta'\rfloor - \lfloor\delta
\rfloor) f(n)$ and $f$ is a function depending on $n$ only.
\end{lem}

In order to derive lower bounds, one can apply a result by Chv\'atal, Cook, and Hartmann~\cite{ChCoHa} that gives sufficient conditions for a sequence of points to be in successive CG closures of a rational polyhedron. The one we provide next is a less general, albeit sufficient for our needs, version of their original lemma.

\begin{lem}\label{lem:iterat-cc}
Let $Q\subseteq \R^n$ be a rational polyhedron, $x \in Q$, $v \in \R^n$, $p \in \N$ and, for $j \in \{1,\dots, p\}$, let $x^j=x-j\cdot v$. Assume that, for all $j \in \{1,\dots,p\}$ and every inequality $cx \leq \delta$ valid for $Q_I$ with $c \in \Z^n$ and $cv < 1$, one has $cx^j\leq \delta$. Then $x^j \in Q^{(j)}$ for all $j \in \{1,\dots, p\}.$
\end{lem}

As a corollary, we have the following result:

\begin{lem}\label{lem:lower}
Let $Q\subseteq \R^n$ be a rational polyhedron, $x \in Q$, and $v \in \Z^n$ be such that $\{x - t v : t \ge 0\} \cap Q_I \neq \emptyset$.
Let $\bar t = \min \{t \ge 0 : x - t v \in Q_I\}$. Then $r(Q)\geq \lceil \;\bar t\;\rceil$.\end{lem}

\begin{proof}
The lemma is trivially true if $\bar t \in [0,1]$, so suppose $\bar t > 1$. By hypothesis, there exists a point $x' \in Q_I$ such that $x=x' + \bar t v$. We apply Lemma \ref{lem:iterat-cc} with $p=\lceil \; \bar t \; \rceil -1$. Let $cx \leq \delta$ be valid for $Q_I$, with $c$ integer. If $cv < 1$, then $cv\leq 0$, since $c$ and $v$ are integer. Then for $j=1,\dots, \lceil \;\bar t\; \rceil -1$, one has $$c x^j=c (x - j \cdot v) =c( x' + (\bar t - j) \cdot v) = c x' + (\bar t-j) cv \leq \delta,$$ where the inequality follows from $x' \in Q_I$, $cv \leq 0$, $\bar t-j > 0$. Hence the hypothesis of Lemma~\ref{lem:iterat-cc} holds. We conclude $x^{\lceil \;\bar t\;\rceil-1} \in Q^{(\lceil \;\bar t \;\rceil -1)}$. Since by construction $x^{\lceil \;\bar t\;\rceil-1} \notin Q_I$, the statement follows.
\end{proof}

\subsection*{Unimodular transformations}

A \emph{unimodular transformation} $u: \R^n \rightarrow \R^n$ maps a point $x \in \R^n$ to $u(x)=Ux + v$, where $U$ is an $n\times n$ unimodular matrix (i.e. a square integer matrix with $|\det(U)|=1$) and $v\in \Z^n$. It is well-known (see e.g.~\cite{sch}) that a nonsingular matrix $U$ is unimodular if and only if so is $U^{-1}$. Furthermore, a unimodular transformation is a bijection of both $\R^n$ and $\Z^n$ that preserves $n$-dimensional volumes. Moreover, the following holds (\cite{EiSc}).

\begin{lem}\label{obs:unimodular-preserves}
Let $Q\subseteq \R^n$ be a polyhedron and $u : \R^n \rightarrow \R^n$, $u(x)=Ux + v$, be a unimodular transformation. Then for each $t \in \N$, an inequality $cx \leq \delta$ is valid for $Q^{(t)}$ if and only if the inequality $cU^{-1}x \leq \delta + cU^{-1}v$ is valid for $u(Q)^{(t)}$. Moreover, the CG rank of $Q$ equals the CG rank of $u(Q)$. \end{lem}

Thanks to the previous lemma, when investigating the CG rank of a $d$-dimensional rational polyhedron $Q\subseteq \R^n$ with $Q\cap\Z^n \neq \emptyset$, we can apply a suitable unimodular transformation and assume that the affine hull of $Q$ is the rational subspace $\{x \in \R^n: x_{d+1}=x_{d+2}=\dots=x_n=0\}$.

\section{Geometric characterization of integral polyhedra with infinite reverse CG rank}\label{sec:main}

In this section we prove Theorem~\ref{th:main}. Since it is already known that, when $P$ is empty, $r^*(P)<+\infty$ (see Lemma \ref{lem:PI=0}), we assume $P\ne\emptyset$.

\begin{Obs}\label{obs:harder-than-you-think}
Let $C\subseteq \R^n$ be a convex set and $v \in \R^n$. Then $\relint(C)+\linhull v=\relint(C+\linhull v)$.
\end{Obs}

\begin{Obs}\label{obs:lf}
Let $C\subseteq\R^n$ be a convex set contained in a rational hyperplane such that $\aff(C) \cap \Z^n \neq \emptyset$.
Then $C$ is relatively lattice-free if and only if there exists $v \in \Z^n \sm \reccone (C)$ such that $C + \linhull v$ is relatively lattice-free.
\end{Obs}

\begin{proof} Since $C$ is contained in a rational hyperplane and $\aff ( C ) \cap \Z^n \neq \emptyset$,  up to unimodular transformations we may assume that $C \subseteq \R^{n-1} \times \{0\}$.
If $C$ is relatively lattice-free, then it is easy to verify that the vector $e^n$ of the standard basis of $\R^n$ is such that $e^n \in \Z^n \sm \reccone (C)$ and $C + \linhull {e^n}$ is relatively lattice-free. Conversely, assume that there exists $v \in \Z^n \sm \reccone (C)$ such that $C + \linhull v$ is relatively lattice-free.
Clearly $\relint (C) \subseteq \relint (C + \linhull v)$, thus $C$ is relatively lattice-free.
\end{proof}

%
%\begin{pf}
%Assume that $C$ is lattice-free.
%Up to unimodular transformations we may assume that $C \subseteq \R^d \times \{0\}$.
%It is easy to verify that the vector $e^n$ of the standard basis of $\R^n$ is such that $e^n \in \Z^n \sm \reccone C$ and $C + \linhull e^n$ is lattice-free.
%
%Assume that there exists $v \in \Z^n \sm \reccone C$ such that $C + \linhull v$ is lattice-free.
%Clearly $\relint C \subseteq \relint (C + \linhull v)$, thus $C$ is lattice-free.
%\end{pf}

\subsection{Proof of Theorem \ref{th:main}: Sufficiency}\label{sec:suff}

Let $P\subseteq \R^n$ be a non-empty integral polyhedron and assume that $P+\linhull v$ is relatively lattice-free for some $v \in \Z^n\setminus \reccone(P)$: we prove that $r^*(P)=+\infty$. Let $\bar x \in \R^n$ be a point in the relative interior of $P$ such that $\bar x+v \notin P$, and $V$ be the set of vertices of $P$. For $\alpha \in \Z_+$, define $Q_{\alpha}=\conv(V,\bar x+\alpha v)+ \reccone(P)$. $Q_{\alpha}$ is a polyhedron and it contains $P$. In order to prove that it is a relaxation of $P$, it suffices to show that $Q_\alpha \cap \Z^n = P \cap \Z^n$. $\bar x + \alpha v \in \relint(P)+\linhull v$ hence, by Observation \ref{obs:harder-than-you-think}, $\bar x + \alpha v \in \relint(P + \linhull v)$. Thus, for each $x \in Q_\alpha$, at least one of the following holds: $x$ lies in $P$; $x$ lies in the relative interior of $P + \linhull v$, and since $P + \linhull v$ is relatively lattice-free by hypothesis, $x$ is not integer. This shows $Q_\alpha \cap \Z^n = P \cap \Z^n$. We now apply Lemma \ref{lem:lower} with $Q=Q_{\alpha}$ and $x=\bar{x}+\alpha v$; note that $\lceil \;\bar t\;\rceil = \alpha$. Hence, we deduce that $r(Q_\alpha)\geq \alpha$. The thesis then follows from the fact that $\alpha$ was chosen arbitrarily in $\Z_+$.

\subsection{Proof of Theorem \ref{th:main}: Necessity}

First, we show that the non-full-dimensional case follows from the full-dimensional one. More precisely, assuming that the statement holds for any full-dimensional polyhedron, we let $P \subseteq \R^n$ be a non-empty integral polyhedron of dimension $d < n$ so that there is no $v \in \Z^n \sm \reccone (P)$ such that $P + \linhull v$ is relatively lattice-free, and we show that $r^*(P)<+\infty$. Hence, let $P$ be as above. Up to a unimodular transformation, we can assume that $\aff(P)=\{x \in \R^n : x_{d+1}=x_{d+2}=\dots=x_n=0\}$. Observation \ref{obs:lf} implies that $P$ is not relatively lattice-free. We then make use of the following fact~\cite[Theorem 1]{non-full-dim}.

\begin{thm}\label{th:main-non-full-dim}
There exists a function $f: \N \rightarrow \N$ such that, for each integral polyhedron $P \subseteq \R^n$ that is not relatively lattice-free, and each relaxation $Q$ of $P$, $Q^{(f(n))}$ is contained in $\aff(P)$.
\end{thm}

By Theorem~\ref{th:main-non-full-dim}, there is an integer $p$ depending only on $n$ such that, for each relaxation $Q$ of $P$, $Q^{(p)}\subseteq\aff(P)$, i.e., modulo at most $p$ iterations of the CG closure, we can assume that both $P$ and $Q$ are full-dimensional, and $P$ is not lattice-free. Hence, $P+\linhull v$ is not lattice-free for any $v \in \Z^d$, and $r^*(P)<+\infty$ follows from the full-dimensional case.

\medskip

Therefore it suffices to show the statement for $P$ full-dimensional. In fact, we show a stronger property, that will be useful later. Let $Ax \leq b$ be an irredundant description of $P$, with $A \in \Z^{m\times n}$ and $b \in \Z^m$. For $k \in \N$, let $P_k=\{x \in \R^n : Ax \leq b + k \cdot {\mathbf{1}}\}$, where $\mathbf1$ denotes the $m$-dimensional all-one vector.

\begin{Prop}\label{prop:full-dim-nec}
Let $P\subseteq \R^n$ be a full-dimensional integral polyhedron, and $Ax \leq b$ an irredundant description of $P$, with $A \in \Z^{m\times n}$ and $b \in \Z^m$. Then the following statements are equivalent.
\begin{enumerate}
\item[(1)] $r^*(P)=+\infty$;
\item[(2)] for each $k \in \N$, there exists a relaxation $Q_k$ of $P$ such that $Q_k \setminus P_k \neq \emptyset$;
\item[(3)] there exists $v \in \Z^n \setminus \reccone(P)$ such that $P+\linhull v$ is lattice-free.
\end{enumerate}
\end{Prop}

\begin{proof} $(3) \Rightarrow (1)$ follows from the sufficiency implication of Theorem \ref{th:main}, which we proved in Section \ref{sec:suff}. In order to show
$(1)\Rightarrow(2)$, suppose to the contrary that there exists some $k \in \N$ such that $Q \subseteq P_{k}$ for each relaxation $Q$ of $P$. Fix any such relaxation $Q$. Then, for each inequality $ax \leq \beta$ from the system $Ax \leq b$ that defines $P$, $ax \leq \beta + k$ is valid for $Q$. Hence, by Lemma \ref{lem:upper}, $ax \leq \beta$ is valid for $Q^{(p)}$, with $p=k f(n)+1$ and $f$ being an appropriate function of $n$ only. This implies that all valid inequalities for $P$ are also valid for $Q^{(p)}$, and consequently $Q^{(p)}=P$. Since $Q$ was taken to be an arbitrary relaxation of $P$, this implies $r^*(P)< + \infty$, contradicting the assumptions.

\smallskip

We are left to prove $(2) \Rightarrow (3)$. This is divided into the following steps: (a) We construct a candidate vector $v \notin \reccone(P)$; (b) We show that $P + \linhull v$ is lattice-free; (c) We show that $v$ can be assumed wlog to be integral.

\medskip

%% SUMMARY - DELETED
%A summary of this step is the following: Claim \ref{cl:P^k} implies that we can find ``arbitrarily big'' relaxations $Q_k$ of $P$. To each $k \in \N$ we associate a vector $v^k$ that connects suitable points of $P$ and $Q_k\setminus P_k$. We show that a subsequence of the normalization of those $v^k$ converges to a vector $v \notin \reccone(P)$.

\noindent {\bf (a) Construction of }$\boldmath{v \notin \reccone(P)}$. By hypothesis, for every $k\in\N$ there exists a point $y^k \in Q_k\setminus P_k$ (see Figure \ref{fig:added}). Let $x^k$ be the point in $P$ such that $d(y^k,
x^k)=d(y^k,P)$ and define $v^k=y^k-x^k$.

\begin{figure}
\centering
% Generated with LaTeXDraw 2.0.8
% Mon Aug 18 23:34:02 CEST 2014
% \usepackage[usenames,dvipsnames]{pstricks}
% \usepackage{epsfig}
% \usepackage{pst-grad} % For gradients
% \usepackage{pst-plot} % For axes
\scalebox{.6} % Change this value to rescale the drawing.
{
\begin{pspicture}(0,-4.22)(16.02,4.22)
\definecolor{color1472b}{rgb}{0.8,0.8,0.8}
\definecolor{color1474b}{rgb}{0.4,0.4,0.4}

\pspolygon[linewidth=0.04,fillstyle=solid,fillcolor=color1472b](11.2,-2.8)(4.0,3.2)(5.2,-0.2)(7.4,-2.8)
\psline[linewidth=0.04cm,arrowsize=0.05291667cm 2.0,arrowlength=1.4,arrowinset=0.4]{<-}(6.0,3.0)(6.0,-4.2)
\psline[linewidth=0.04cm,arrowsize=0.05291667cm 2.0,arrowlength=1.4,arrowinset=0.4]{->}(3.2,-1.8)(13.4,-1.8)
\pspolygon[linewidth=0.04,fillstyle=solid,fillcolor=color1474b](6.0,0.2)(8.0,-1.8)(10.0,-1.8)
\psdots[dotsize=0.4](6.0,-1.8)
\psdots[dotsize=0.4](6.0,0.2)
\psdots[dotsize=0.4,dotangle=-90.0](8.0,-1.8)
\psdots[dotsize=0.4,dotangle=-90.0](10.0,-1.8)
\psdots[dotsize=0.4,dotangle=-90.0](12.0,-1.8)
\psdots[dotsize=0.4,dotangle=-90.0](8.0,0.2)
\psdots[dotsize=0.4,dotangle=-90.0](10.0,0.2)
\psdots[dotsize=0.4,dotangle=-90.0](8.0,-3.8)
\psdots[dotsize=0.4,dotangle=-90.0](10.0,-3.8)
\psdots[dotsize=0.4,dotangle=-90.0](12.0,-3.8)
\psdots[dotsize=0.4](4.0,-1.8)
\psdots[dotsize=0.4](4.0,0.2)
\psline[linewidth=0.04cm,linestyle=dashed,dash=0.16cm 0.16cm](0.0,4.2)(8.0,-3.8)
\psline[linewidth=0.04cm,linestyle=dashed,dash=0.16cm 0.16cm](8.0,-3.8)(16.0,-3.8)
\psdots[dotsize=0.4](6.0,0.2)
\psdots[dotsize=0.4](6.0,2.2)
\psdots[dotsize=0.4,dotangle=-90.0](8.0,0.2)
\psdots[dotsize=0.4,dotangle=-90.0](10.0,0.2)
\psdots[dotsize=0.4,dotangle=-90.0](8.0,2.2)
\psdots[dotsize=0.4,dotangle=-90.0](10.0,2.2)
\psline[linewidth=0.04cm,linestyle=dashed,dash=0.16cm 0.16cm](16.0,-3.8)(0.0,4.2)
\psdots[dotsize=0.4,dotangle=-90.0](4.0,-1.8)
\psdots[dotsize=0.4,dotangle=-90.0](2.0,2.2)
\psdots[dotsize=0.4,dotangle=-90.0](4.0,2.2)

\pscircle[linewidth=0.04,linestyle=dashed,dash=0.16cm 0.16cm,dimen=outer,fillstyle=solid,fillcolor=black](4.5,2.3){0.1}
\usefont{T1}{ptm}{m}{n}
\rput(4.9153125,2.14){\Large $y_1$}

\end{pspicture}
}
\caption{Illustration from part (a) of the proof of Proposition \ref{prop:full-dim-nec}. In dark grey, the polytope $P$ described by inequalities $-x_2 \leq 0$; $-x_1-x_2 \leq -1$; $x_1 + 2x_2 \leq 2$. In dashed lines, polytope $P_1$. In light grey, polytope $Q_1$.}\label{fig:added}
\end{figure}
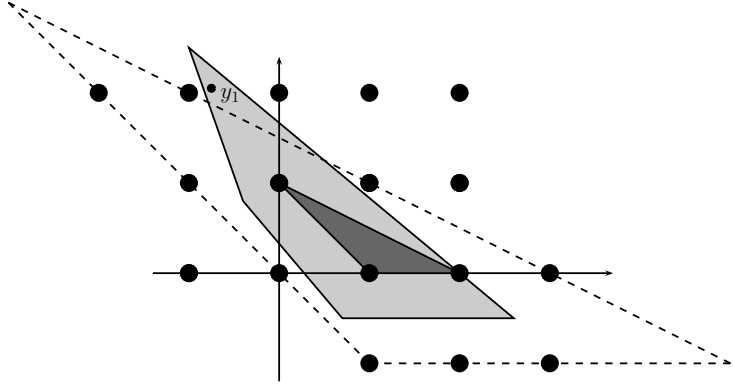

%Moreover, since $P \subseteq Q_k$ and $Q_k$ is
%convex, we can always choose $\tilde y^k$ to be at distance at most
%$k+1$ from $P$. Summing up, we choose $\tilde y^k \in \interior(Q_k)$ with
%$\frac{k}{\alpha}< d(\tilde y^k,P)< k + 1 $.

\begin{Rm}\label{rm:normal}
For every $k\in\N$, the hyperplane $H=\{x \in \R^n : v^k x = v^kx^k\}$ is a supporting hyperplane
for $P$ containing $x^k$.
\end{Rm}

Consider the sequence of normalized vectors $\big\{\frac{v^k}{\|v^k\|}\big\}_{k \in \N}$. Since it is contained in the $(n-1)$-dimensional unit sphere $S$, which is a compact set, it has a subsequence that converges to an element of $S$, say $v$. We denote by ${\cal I}$ the set of indices of this subsequence.
Remark~\ref{rm:normal} shows that every vector $v^k$ belongs to the {\em optimality cone} of $P$, which is defined as the set of vectors $c$ such that the problem $\max\{cx:x\in P\}$ has finite optimum. Since the optimality cone of a polyhedron is a polyhedral cone, in particular it is a closed set. Then $v$ belongs to the optimality cone of $P$. This implies that $v\notin\reccone(P)$, as $\max\{cx:x\in P\}$ is never finite if $c$ is a non-zero vector in $\reccone(P)$.

\medskip

% SUMMARY DELETED
%A summary of this step is as follows: We assume by contradiction that $\tilde z \in \Z^n$ for some $\tilde z \in \interior(P + \linhull v)$. We first construct suitable $w \in \interior(P)$, $\beta > 0$, and $z \in \Z^n$, such that $z=w + \beta v$. We then prove that, for any $k \in {\cal I}$ big enough, a suitable integer translation of $z$ belongs to $Q_k\setminus P$. This gives the desired contradiction, since each $Q_k$ is a relaxation of $P$.
%
%\smallskip

\noindent {\bf (b)} $\boldmath{ P+\linhull v}$ {\bf is lattice-free}. Assume the existence of $\tilde z \in \Z^n$ for some $\tilde z \in \interior(P + \linhull v)$. Observation \ref{obs:harder-than-you-think} implies that there exist $\tilde w \in \interior(P)$ and $\alpha \in \R$ such that $\tilde z = \tilde w + \alpha v$. Since $P$ is a rational polyhedron, $P=P^* + \intcone(R)$, where $R=\{r^1,\dots,r^{|R|}\}$ is a set of integer generators of $\reccone(P)$ and $P^*$ is a suitable polytope such that $\tilde w \in \interior(P^*)$ (for instance, if we let $V$ be the vertex set of $P$, we can take $P^*=\conv\{V \cup \bigcup_{i=1}^{|R|} (\tilde w + r^i)\}$). We denote by $\delta$ the \emph{geometric diameter} of $P^*$, i.e. the maximum distance between two points of $P^*$.

\begin{cl}\label{cl:cos-theta}
There exist a number $\beta > 2\delta$ and points $w \in \interior(P^*)$ and $z \in \Z^n$, such that $z=w + \beta v$.
\end{cl}

\begin{cpf}
%Suppose first that $v$ is rational. Since $\tilde y$ is integral, there exists $\alpha' > 2(\delta + 1) - \alpha$ such that $\tilde y + \alpha 'v \in \Z^n$. Hence, we can set $\bar x = \tilde x$, $\beta = \alpha' + \alpha>  2(\delta + 1)$, and $\bar y = \tilde y + \alpha' v = \tilde x +(\alpha' + \alpha) v = \bar x + \beta v$.
%Now suppose that $v \notin \Q^n$.
We make use of the following fact, shown by Basu et al.~\cite[Lemma 13]{BasuDirichelet} as a consequence of the well-known Dirichlet's approximation theorem: \emph{Given $u \in \Z^n$ and $r\in \R^n$,  then for every  $\varepsilon > 0$ and $\bar \lambda \geq 0$, there exists an integer point at distance less than $\varepsilon$ from the halfline $\{u+ \lambda r : \lambda \geq \bar \lambda\}$.} Apply this result with $u=\tilde z$, $r=v$, $0<\varepsilon<d(\tilde w, \bd(P^*))$, and $\bar \lambda = \max(0,2\delta - \alpha + \varepsilon)$. It guarantees the existence of an integer point $z$ at distance less than $\varepsilon$ from the halfline $\{\tilde z + \lambda v : \lambda \geq 2\delta - \alpha + \varepsilon\} = \{\tilde w + \lambda v : \lambda \geq 2\delta + \varepsilon\}$. Then $z= w + \beta v$ for some point $w$ at distance less than $\varepsilon$ from $\tilde w$ and $\beta > 2\delta $. As $\varepsilon<d(\tilde w, \bd(P^*))$, it follows that $w \in \interior(P^*)$. \end{cpf}

Let $\beta$, $w$, $z$ be as in Claim \ref{cl:cos-theta}. If for $a \in \Z^{|R|}_+$ we define $P^*(a)=P^* + \sum_{i=1,\dots,|R|} a_ir^i$, then $P=\bigcup_{a \in \Z^{|R|}_+} P^*(a)$ (see Figure \ref{fig:thr-(b)}). Recall that, for $k \in \N$, one has $y^k \in Q_k\setminus P_k$, $x^k \in P$, and $v^k=y^k-x^k$. For $k \in \N$, let $a^k \in \Z^{|R|}_+$ be such that $x^k \in P^*(a^k)$. Also, let $w^k= w + \sum_{i=1}^{|R|}a^k_ir^i $. Note that each $w^k$ is a translation of $w$ by an integer combination of integer vectors $r^1,\dots, r^{|R|}$, so that $w^k$ lies in the same translation of $P^*$ as $x^k$. This implies $d(w^k,x^k)\leq \delta$. For each $k \in \N$, we also define $z^k = w^k + \beta v$. One easily checks that $z^k= z + \sum_{i=1}^{|R|}a^k_i r^i$, that is, $z^k$ is a translation of $z$ by integer vectors $r^1,\dots, r^{|R|}$ with the same multipliers as $w^k$, hence it is an integer vector. The proof of (b) is an immediate consequence of the following claim, which contradicts the fact that $Q_k$ is a relaxation of $P$ for every $k \in \N$.

\begin{figure}
\centering
% Generated with LaTeXDraw 2.0.8
% Thu Apr 19 12:54:50 CEST 2012
% \usepackage[usenames,dvipsnames]{pstricks}
% \usepackage{epsfig}
% \usepackage{pst-grad} % For gradients
% \usepackage{pst-plot} % For axes
\scalebox{.7} % Change this value to rescale the drawing.
{
\begin{pspicture}(0,-3.331)(11.79875,3.33)
\definecolor{color47b}{rgb}{0.7,0.7,0.7}
\definecolor{color1648b}{rgb}{0.9,0.9,0.9}
\definecolor{color1656b}{rgb}{0.5,0.5,0.5}

%% Polytope

\psline[linewidth=0.05](13,2.4)(5,0)(5,-3)(13,-3)

%% P^*

\psline[linewidth=0.02,opacity=0.5,linecolor=color47b,linestyle=dashed,dash=0.16cm 0.16cm,fillstyle=solid,fillcolor=color47b](9,1.2)(7,0.6)(7,-2.4)(9,-2.4)
\psline[linewidth=0.0020,opacity=0.5,linecolor=color47b,linestyle=dashed,dash=0.16cm 0.16cm,fillstyle=solid,fillcolor=color1648b](7,-3)(5,-3)(5,0)(7,0.6)
\psline[linewidth=0.02,opacity=0.5,linecolor=color47b,linestyle=dashed,dash=0.16cm 0.16cm,fillstyle=solid,fillcolor=color1656b](9,1.2)(11,1.8)(11,-1.8)(9,-1.8)

%% Rays

\psline[linewidth=0.05cm,arrowsize=0.15cm 2.0,arrowlength=1.4,arrowinset=0.4]{->}(0.6,-1.17)(2.6,-0.59)
\psline[linewidth=0.05cm,arrowsize=0.15cm 2.0,arrowlength=1.4,arrowinset=0.4]{->}(0.6,-2.41)(2.6,-2.41)

\usefont{T1}{ptm}{m}{n}
\rput(0.40421876,-2.515){\large $r_1$}
\usefont{T1}{ptm}{m}{n}
\rput(0.36421874,-1.175){\large $r_2$}

%% \bar x^k with arrows

\psellipse[linewidth=0.04,dimen=outer,fillstyle=solid,fillcolor=black](6.606875,-2.5)(0.09,0.09) %% \bar x
\psellipse[linewidth=0.04,dimen=outer,fillstyle=solid,fillcolor=black](8.686875,-1.84)(0.09,0.09) %% \bar x ^1
\psellipse[linewidth=0.04,dimen=outer,fillstyle=solid,fillcolor=black](10.766875,-1.22)(0.09,0.09) %% \bar x ^2

\psline[linewidth=0.05cm,arrowsize=0.15cm 2.0,arrowlength=1.4,arrowinset=0.4]{->}(6.696875,-2.45)(8.616875,-1.87) %% \bar x -> \bar x ^1
\psline[linewidth=0.05cm,arrowsize=0.15cm 2.0,arrowlength=1.4,arrowinset=0.4]{->}(8.77,-1.8)(10.77,-1.2)%% \bar x^1 -> \bar x^2

\usefont{T1}{ptm}{m}{n}
\rput(6.2,-1.5){$\beta v$}

\usefont{T1}{ptm}{m}{n}
\rput(6.688281,-2.82){$w$}
\usefont{T1}{ptm}{m}{n}
\rput(8.738282,-2.14){$w^1$}
\usefont{T1}{ptm}{m}{n}
\rput(10.678281,-1.54){$w^2$}

%\psline[linewidth=0.05cm,arrowsize=0.05291667cm 2.0,arrowlength=1.4,arrowinset=0.4]{->}(8.77,2.6)(10.77,3.2)

%% x^1, x^2

\psellipse[linewidth=0.04,dimen=outer,fillstyle=solid,fillcolor=black](7.746875,0.85)(0.09,0.09)
\psellipse[linewidth=0.04,dimen=outer,fillstyle=solid,fillcolor=black](9.3,1.3)(0.09,0.09)

\usefont{T1}{ptm}{m}{n}
\rput(9.5,0.95){$x^2$}
\usefont{T1}{ptm}{m}{n}
\rput(7.7682815,0.5){$x^1$}

%% \bar y^k with arrows

\psellipse[linewidth=0.04,dimen=outer,fillstyle=solid,fillcolor=black](6.626875,1.92)(0.09,0.09)
\psellipse[linewidth=0.04,dimen=outer,fillstyle=solid,fillcolor=black](8.686875,2.6)(0.09,0.09)
\psellipse[linewidth=0.04,dimen=outer,fillstyle=solid,fillcolor=black](10.766875,3.24)(0.09,0.09)

\psline[linewidth=0.05cm,linestyle=dashed,dash=0.16cm 0.16cm,arrowsize=0.15cm 2.0,arrowlength=1.4,arrowinset=0.4]{->}(6.626875,-2.3)(6.626875,1.87)  %% \bar x -> \bar y

\psline[linewidth=0.05cm,linestyle=dashed,dash=0.16cm 0.16cm,arrowsize=0.15cm 2.0,arrowlength=1.4,arrowinset=0.4]{->}(8.686875,-1.84) (8.686875,2.5) %% \bar x^1 -> \bar y^1

\psline[linewidth=0.05cm,linestyle=dashed,dash=0.16cm 0.16cm,arrowsize=0.15cm 2.0,arrowlength=1.4,arrowinset=0.4]{->}(10.766875,-1.22) (10.766875,3.1) %% \bar x^1 -> \bar y^1

%\psline[linewidth=0.05cm,arrowsize=0.05291667cm 2.0,arrowlength=1.4,arrowinset=0.4]{->}(6.696875,1.97)(8.616875,2.55)

\usefont{T1}{ptm}{m}{n}
\rput(6.978281,1.78){$z$}
\usefont{T1}{ptm}{m}{n}
\rput(9.008282,2.46){$z^1$}
\usefont{T1}{ptm}{m}{n}
\rput(11.1,3.14){$z^2$}
\usefont{T1}{ptm}{m}{n}
\end{pspicture}
}
\caption{Illustration from part (b) of the proof of Proposition \ref{prop:full-dim-nec}. On the left: the vectors $r^1$ and $r^2$ from $\reccone(P)$. On the right: polytope $P$, and its covering with polyhedra $P^*(a)$, $a \in\Z_+$. In increasingly darker shadows of grey: $P^*=P^*{0 \choose 0}$, $P^*{0 \choose 1}$, $P^*{0 \choose 2}$. If moreover $x^1 \in P^*{0 \choose 1}$ and $x^2 \in P^*{0 \choose 2}$, we obtain $w^1$, $w^2$, $z^1$, $z^2$ as in the picture.}\label{fig:thr-(b)}
\end{figure}
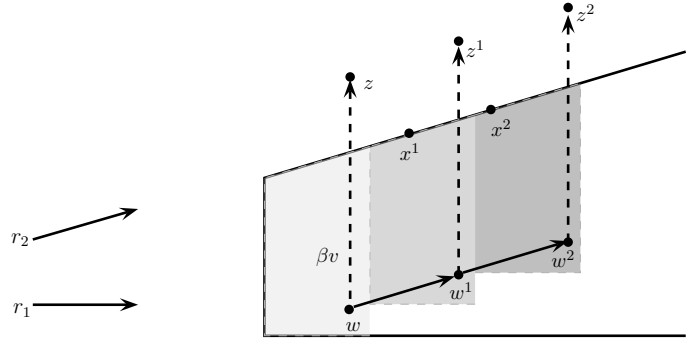

\begin{cl}\label{cl:final-full}
$z^k \in Q_k \setminus P$ for each $k \in {\cal I}$ large enough.
\end{cl}

%This concludes the proof of (b).

\begin{cpf}
We first show that $z^k \notin P$ for $k \in {\cal I}$ large enough. By Remark \ref{rm:normal}, the hyperplane $H=\{ x \in \R^n : v^k x = v^kx^k\}$ is a supporting hyperplane of $P$ containing $x^k$. Let $\gamma \in \R$ be such that $w^k + \gamma \frac{v^k}{\|v^k\|} \in H$. Note that $\gamma$ is well-defined since $v^k$ is normal to $H$, and moreover $\gamma \geq 0$, as $w^k \in P$. Since $\frac{v^k}{\|v^k\|}$ is a unit vector normal to $H$, one has
\begin{equation}\label{eq:gamma}
\gamma = d(w^k, H) \leq d(w^k, x^k) \leq \delta,
\end{equation}
where the first inequality comes from the fact that $x^k \in H$.

Let now $\phi$ be the angle between $v^k$ and $v$, and $k \in {\cal I}$ be large enough, so that $0 \leq \phi \leq \frac{\pi}{3}$. Let $ \sigma \in \R$ be such that $w^k + \sigma v \in H$ (recall that $\|v\|=1$). By simple trigonometric arguments and by \eqref{eq:gamma}, we obtain $\sigma = \frac{\gamma}{\cos \phi} \leq 2 \delta$. Hence, points $w^k + \lambda v$ with $\lambda > 2\delta $ do not belong to $P$. In particular, $z^k \notin P$, since $z^k = w^k + \beta v$ with $\beta > 2\delta$ from Claim \ref{cl:cos-theta}.

\smallskip

We now show that $z^k \in Q_k$ for $k \in {\cal I}$ large enough.  Let $\varepsilon$ be such that $0<\varepsilon<d(w,\bd(P^*))$. Note that $\varepsilon<d(w^k,\bd(P))$ for all $k \in \N$.  For each $k\in\N$, let $H^k$ be the hyperplane with normal $v$ containing point $w^k$, i.e., $H^k=\{x:vx=vw^k\}$. Define $B^k$ to be the $(n-1)$-ball of radius $\varepsilon$ lying on $H^k$ and centered at $w^k$. Note that $B^k\subseteq P$ and $z^k-B^k$ is the $(n-1)$-ball of radius $\varepsilon$ centered at $\beta v$ and lying on the hyperplane $\{x : vx=\beta\}$. Hence the cone $C$ generated by $\{z^k- x:x\in B^k\}$ does not depend on $k$, and it is indeed a \emph{cone of revolution} defined by direction $v$ and some angle $0<2\theta<\pi/2$ (see Figure \ref{fig:thr-1(b)bis}), i.e. $C$ is the set of vectors of $\R^n$ that form an angle of at most $2\theta$ with $v$. Note that \begin{equation}\label{eq:star}z^k\in\conv(x,B^k)\quad \hbox{for every $x\in z^k+C$}.\end{equation} Now let $D$ be the cone of revolution of direction $v$ and angle $\theta$. Note that $D$ is strictly contained in cone $C$. Since $d(x^k,w^k)\le\delta$ for all $k$, there exists a positive number $\tau$ such that $\{x\in x^k+D: d(x,x^k)\ge\tau\}\subseteq z^k+C$ for all $k \in \N$. Since $\lim_{k\rightarrow + \infty} d(y^k,P) = + \infty$, for $k \in \N$ large enough $d(y^k,x^k)= d(y^k,P)\ge \tau$. If moreover we take $k\in\cal I$ large enough so that the angle between $v^k$ and $v$ is at most $\theta$, one has $y^k \in x^k+D$ and consequently $y^k\in z^k+C$. Because $y^k \in Q_k$ and \eqref{eq:star}, we conclude that $z^k \in Q_k$, as required.
\end{cpf}

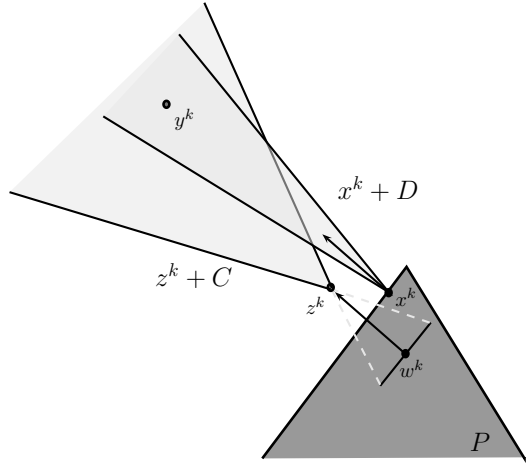
\begin{figure}
\centering% Generated with LaTeXDraw 2.0.8
% Tue May 15 20:31:50 CEST 2012
% \usepackage[usenames,dvipsnames]{pstricks}
% \usepackage{epsfig}
% \usepackage{pst-grad} % For gradients
% \usepackage{pst-plot} % For axes
% Generated with LaTeXDraw 2.0.8
% Wed May 30 18:30:31 CEST 2012
% \usepackage[usenames,dvipsnames]{pstricks}
% \usepackage{epsfig}
% \usepackage{pst-grad} % For gradients
% \usepackage{pst-plot} % For axes
\scalebox{.7} % Change this value to rescale the drawing.
{
\begin{pspicture}(0,-4.3554187)(9.805,4.350419)
\definecolor{color286b}{rgb}{0.6,0.6,0.6}
\definecolor{color289b}{rgb}{0.9,0.9,0.9}
\psline[linewidth=0.05,fillstyle=solid,fillcolor=color286b](4.8,-4.330419)(7.54,-0.68958116)(9.78,-4.330419)(9.78,-4.3095813)
\psline[linewidth=0.04,fillstyle=solid,opacity=.5,fillcolor=color289b](0.0,0.73041886)(6.14,-1.1295811)(3.7,4.330419)
\psline[linewidth=0.04,fillstyle=solid,opacity=.5,fillcolor=color289b](1.78,2.1704187)(7.2,-1.1695812)(3.22,3.730419)
\usefont{T1}{ptm}{m}{n}
\rput(8.935312,-4.009581){\Large $P$}
\psellipse[linewidth=0.05,dimen=outer,fillstyle=solid,fillcolor=black](7.52,-2.3395813)(0.08,0.09)
\usefont{T1}{ptm}{m}{n}
\rput(7.641406,-2.6195812){$w^k$}
\usefont{T1}{ptm}{m}{n}
\rput(5.811406,-1.4195812){$z^k$}
\psline[linewidth=0.04cm,arrowsize=0.05291667cm 2.0,arrowlength=1.4,arrowinset=0.4]{->}(7.48,-2.309581)(6.18,-1.1695812)
\psline[linewidth=0.04cm](8.0,-1.7495811)(7.04,-2.9495811)
\psline[linewidth=0.04cm,linecolor=color289b,linestyle=dashed,dash=0.16cm 0.16cm](7.98,-1.7495811)(6.1,-1.1095811)
\psline[linewidth=0.04cm,linecolor=color289b,linestyle=dashed,dash=0.16cm 0.16cm](7.04,-2.9295812)(6.12,-1.1295811)
\psellipse[linewidth=0.05,dimen=outer,fillstyle=solid,fillcolor=black](7.2,-1.1795812)(0.08,0.09)
\usefont{T1}{ptm}{m}{n}
\rput(7.5314063,-1.2595811){$x^k$}
\psellipse[linewidth=0.05,dimen=outer,opacity=.5,fillstyle=solid,fillcolor=black](6.1,-1.0795811)(0.08,0.09)
\psellipse[linewidth=0.05,dimen=outer,opacity=.5,fillstyle=solid,fillcolor=black](2.98,2.4004188)(0.08,0.09)
\usefont{T1}{ptm}{m}{n}
\rput(3.3214064,2.1204188){$y^k$}
\usefont{T1}{ptm}{m}{n}
\rput(3.5,-.85){\Large $z^k+C$}
\usefont{T1}{ptm}{m}{n}
\rput(7,.8){\Large $x^k+D$}
\psline[linewidth=0.04cm,arrowsize=0.05291667cm 2.0,arrowlength=1.4,arrowinset=0.4]{->}(7.24,-1.2095811)(5.94,-0.06958115)
\end{pspicture}

}
\caption{Illustration from the proof of Claim \ref{cl:final-full}. Both $C$ and $D$ are cones of revolution defined by direction $v$, with angles respectively $2\theta$ and $\theta$.}\label{fig:thr-1(b)bis}\end{figure}

\medskip

\noindent {\bf (c)} $\boldmath{v}$ {\bf can be assumed wlog to be integral}. In~\cite[Theorem 2]{BasuDirichelet} (see also~\cite{Lov}) it is proved that a maximal lattice-free convex set is either an irrational affine hyperplane of $\R^n$, or a polyhedron $Q + L$, where $Q$ is a polytope and $L$ is a rational linear space. $P+\linhull v$ is lattice-free, thus it is contained in a maximal lattice-free convex set. Since it is full-dimensional, it is not contained in an irrational hyperplane. It follows that $P \subseteq Q+L$, with $Q,L$ as above. Moreover, $L$ has dimension at least $1$, since it contains $v$. Pick a set $S\subseteq \Z^n$ of generators of $L$ such that $v$ belongs to the cone generated by $S$. Since $v \notin \reccone(P)$, then $s \notin \reccone(P)$ for at least one $s \in S$. Moreover, $P+\linhull s\subseteq Q+L$ and it is full-dimensional, hence it is lattice-free. We can then replace $v$ by $s$. This concludes the proof of Proposition \ref{prop:full-dim-nec} and Theorem~\ref{th:main}.\end{proof}

\section{On some polyhedra with finite reverse CG rank}\label{sec:cg-more}

In this section we investigate the behavior of the reverse CG rank for two classes of polyhedra. Namely, let ${\cal A}$ be the family of integral polyhedra $P$ such that $(i)$ no facet of $P$ is relatively lattice-free and $(ii)$ either $P$ is not relatively lattice-free or $P$ is full-dimensional; also, let ${\cal B}$ be the family of integral polyhedra that are not relatively lattice-free. We show the following.

\begin{thm}\label{thr:bounds}
\begin{enumerate}[\upshape(i)]
\item For each $n \in \N$,
$\sup\{r^*(P): P \subseteq \R^n, P \in {\cal A}\}\leq \lambda(n)$, where $\lambda$  is a function depending on $n$ only.
\item For each $n, k \in \N$, $\sup\{r^*(P): P \subseteq \R^n, P \in {\cal B}, |\relint(P) \cap \Z^n| \leq k\}\leq \mu(n,k)$, where $\mu$ is a function depending on $n$ and $k$ only.
\end{enumerate}
\end{thm}

We build on the following result~\cite[Theorem 12]{non-full-dim}.

\begin{thm}\label{thm:ellipsoid}
There exists a function $\phi: \N \rightarrow \R_+\setminus \{0\}$ such that every integral non-lattice-free polyhedron $P\subseteq\R^n$ contains a centrally symmetric polytope of volume $\phi(n)$, whose only integer point is its center.
\end{thm}

%\begin{pf}
%Let $P\subseteq \R^n$ be an unbounded integral non-lattice-free polyhedron: we show it contains an integral non-lattice-free polytope of the same dimension. The statement then follows by Theorem \ref{thr:ellipsoid}. It is well-known (see e.g.~\cite{sc}) that $P=\conv(S) +\cone (R)$, where $S=\{s^1,\dots, s^{|S|}\}$ is a set made of exactly one vector from each minimal face of $P$, whereas $R=\{r^1,\dots, r^{|R|}\}$ is the set of extreme rays of $\reccone(R)$, which we can assume to be integer. We can also suppose that vectors from $S$ are integer, for $P$ is integral. Let $\bar x \in \interior(P)\cap \Z^n$ be obtained with multipliers $\mu_{1},\dots,\mu_{|S|}$ and $\gamma_{1},\dots,\gamma_{|R|}$. Note that we can assume that all multipliers are strictly positive and, by appropriate scaling, that $\mu_{1}> \sum_{i=1}^{|R|}\gamma_i$. Let now $S'= S \cup \{s^1+r^i\}_{i=1}^{|R|}$. One readily verifies that $P'=\conv(S')$ is integral, $\bar x \in \relint(P')$ and $\aff(P')=\aff(P)$, concluding the proof.
%\end{pf}

The proof of the following lemma uses Minkowski's Convex Body Theorem in a way similar to the proof of~\cite[Theorem 1]{non-full-dim}.

\begin{lem}\label{lem:aho}
Let $P\subseteq \R^n$ be an integral polyhedron. Let $cx \leq \delta$ be a valid inequality for $P$ inducing a facet of $P$ that is not relatively lattice-free. Then, for every relaxation $Q$ of $P$ contained in $\aff(P)$, $cx\le\delta$ is valid for $Q^{(p)}$, where $p$ depends on $n$ only.
\end{lem}

\begin{proof}
Let $P$ be $d$-dimensional and $F$ be the facet of $P$ induced by inequality $cx\le\delta$. If $d=0$, then there is nothing to prove, as $P$ has no facet. If $d=1$, then $Q^{(1)}=P$ since $Q$ is a relaxation of $P$ contained in $\aff(P)$. Hence we assume $d \geq 2$. Modulo a unimodular transformation, we can assume that $\aff(P)=\{x\in\R^n:x_{d+1}=\dots=x_n=0\}$ and $cx \leq \delta$ is the inequality $x_d \leq 0$. By Theorem \ref{thm:ellipsoid}, $F$ contains a $(d-1)$-dimensional centrally symmetric polytope $E$ of volume $\phi(d-1)$, whose only integer point is its center. We assume wlog that this point is the origin. We now argue that the inequality $x_d \leq \bar \delta = \max\{\frac{i\cdot 2^{i-1}}{\phi(i-1)}:2\le i\le n\}$ is valid for each relaxation $Q$ of $P$ contained in $\aff(P)$. Note that $\bar \delta$ only depends on $n$. Assume by contradiction that there exists a point $\bar x \in Q$ with $\bar x_d > \bar \delta \geq \frac{d\cdot 2^{d-1}}{\phi(d-1)}$. Define $C=\conv(E,\bar x)\subseteq Q$. Since $Q$ is a relaxation of $P$, $C$ is a $d$-dimensional convex body whose only integer point is the origin, which lies on its boundary. Moreover,
\[\vol(C)=\bar x_d \cdot \frac{\vol(E)}{d}> \frac{d\cdot 2^{d-1}}{\phi(d-1)} \cdot \frac{\phi(d-1)}{d} = 2^{d-1}.\]
Let $C'$ be the symmetrization of $C$ w.r.t.\ the origin, i.e. $C'=C \cup -C$. Note that $C'$ is a $d$-dimensional centrally symmetric polytope in the space of the first $d$ variables whose only integer point is the origin. Furthermore, $\vol(C')=2 \vol(C) > 2^d$. However, by Minkowski's Convex Body Theorem (see, e.g.,~\cite{Bar}), every centrally symmetric convex body in $\R^d$ whose only integer point is the origin has volume at most $2^{d}$. This is a contradiction. Therefore the inequality $x_d \leq \bar \delta$ is valid for each relaxation $Q$ of $P$ contained in $\aff(P)$. Lemma \ref{lem:upper} then implies that $cx\le\delta$ is valid for $Q^{(p)}$, where $p$ depends only on $n$.
\end{proof}

\noindent\emph{Proof of Theorem \ref{thr:bounds}.}
$(i)$. Let $P \in {\cal A}$. Then no facet of $P$ is relatively lattice-free. Suppose first that $P$ is full-dimensional. By Lemma \ref{lem:aho}, for each facet-defining inequality $cx \leq \delta$ of $P$, $cx \leq \delta$ is valid for $Q^{(p)}$, with $p$ depending on $n$ only. This implies that $Q^{(p)}=P$, concluding the proof. Now, assume that $P$ is of dimension $d<n$. By definition of ${\cal A}$, $P$ is not relatively lattice-free. Theorem \ref{th:main-non-full-dim} implies that there exists a number $p$ depending only on $n$ such that, for each relaxation $Q$ of $P$, $Q^{(p)}\subseteq\aff(P)$. Thus in a number of iterations of the CG closure depending on $n$ only we are back to the full-dimensional case. This proves $(i)$.

\smallskip

\noindent $(ii)$. Now fix $n,k \in \N$, $k\geq 1$, and consider the family of polyhedra $P\subseteq \R^n$, $P \in {\cal B}$, with $|\relint(P) \cap \Z^n| =k$. Actually, this family is only composed of polytopes, as every unbounded integral polyhedron with an integer point in its relative interior contains infinitely many of those. By Theorem \ref{th:main}, $r^*(P)$ is finite for each polytope from this family. Lagarias and Ziegler~\cite{Lazi} showed that, up to unimodular transformations, for each $d$ and $k\ge1$ there is only a finite number of $d$-dimensional polytopes with $k$ integer points in their relative interior. Hence there exists a number $t_{n,k}$ such that $r^*(P)\le t_{n,k}$ for all polytopes $P \subseteq \R^n$ with $P \in {\cal B}$ and $|\relint(P) \cap \Z^n| =k$, concluding the proof of $(ii)$.\hspace*{\stretch{1}}$\square$

\medskip

Theorem~\ref{thr:bounds} shows that full-dimensional lattice-free integral polyhedra with an integer point in the relative interior of each facet have finite reverse CG rank. For these polyhedra, the non-existence of a direction $v$ as in the statement of Theorem~\ref{th:main} is due to the fact that by applying any direction $v\notin\reccone(P)$, one of the integer points of the polyhedron (more precisely, one of those lying in the relative interior of the facets) will fall in the interior of $P+\linhull v$. However, this is not a necessary condition for an integral polytope to have finite reverse CG rank. As an example, consider the polytope $P=\conv\{(0,0,0),$ $(3,1,0), (2,3,0),(3,2,2)\}\subseteq\R^3$. If we take, e.g.,  $v=(1,0,0)$, then $P+\linhull v$ does not contain any integer point of $P$ in its interior, but $P+\linhull v$ is not lattice-free, as $(3,2,1)$ is in its interior.

\smallskip

As a counterpart to Theorem \ref{thr:bounds}, we now provide examples of families of polytopes from $\cal A$ (resp. $\cal B$) where $r^*$ grows with the dimension of the ambient space (resp. with the number of integer points in the relative interior of the polytopes). Indeed, let $P\subseteq \R^n$ be an integral $d$-dimensional polytope, with $d \leq n-1$. Up to a unimodular transformation, $P$ is contained in the hyperplane defined by the equation $x_n=0$. Bockmayr et al.~\cite{Bock} showed that, for each $k \in \N$, there exists a polyhedron $Q_k\subseteq [0,1]^k$ such that $r(Q_k)=k$ and $(Q_k)_I=\emptyset$. Let $\bar Q=\{(x,1): x \in Q_{n-1}\}\subseteq \R^n$, and set $Q=\conv(P,\bar Q)$. Note that $Q$ is a relaxation of $P$. Since $\bar Q^{(t)}\subseteq Q^{(t)}$, one has $r(Q)\geq r(\bar Q)=n-1$. Moreover, as $Q_I=P$, one has that for each polyhedron in $\R^n$ that is not full-dimensional, $r^*(P)\geq n-1$. Hence, any bound on $r^*(P)$ must grow with the dimension of the space $P$ lives in. (Incidentally, this also implies that it makes no sense to study the parameter $r^*(P)$ when $P$ is viewed as an integral $d$-polyhedron that can be embedded in {\em any} real space of dimension $n\geq d$, since this value is equal to $+\infty$ for each integral polytope $P$).

For $k \in \N$, consider the integral polytope $P_k\subseteq \R^2$ defined by the following system of inequalities:

$$\begin{array}{lcrcll}
x_1 &  & & \geq & 0 \\
& & x_2 & \geq & 0 \\
& & x_2 & \leq & k \\
x_1 & - & \frac{1}{k} x_2 & \leq & 1 \\
\end{array}$$

Note that $P_k$ has $k-1$ integer points in its interior. Let $Q_k=\conv(P_k,\bar x)$, where $\bar x= (1/2,-k/2)$. Clearly $Q_k$ is a relaxation of $P$. Using Lemma \ref{lem:lower}, we obtain $r(Q_k)\geq k/2$. This implies that any upper bound on $r^*(P)$ for $P \in {\cal B}$ must grow with the number of integer points in the interior of $P$. One immediately extends these results to unbounded polyhedra and higher dimensions.

\section{Algorithmic issues}\label{sec:algo}

Unfortunately, Theorem \ref{th:main} does not seem to immediately imply an algorithm for detecting if an integral polyhedron has finite reverse CG rank. In this section, we shed some light on this problem. We employ some standard definitions and notation from complexity theory, see e.g.~\cite{Ho}, and from polyhedral theory, see e.g. \cite{sch}. All reductions that we give between decision problems are Karp reductions, and the classes of $\np$-complete, $\np$-hard, etc. problems are those defined accordingly. All results that are assumed as known in this section are also standard and can be found in (at least one of) \cite{KaPf,sch}. Detecting if an input ${\cal H}$-polyhedron has infinite reverse CG rank can be stated as the following decision problem.

\begin{framed}
(\RCGR)
\begin{quote}%(\RCGR)

  Given: an integral polyhedron $P=\{x \in \R^n: Ax \leq b\}$, where $A \in \Z^{m \times n}$ and $b \in \Z^m$;

  Decide: if there exists $v \in \Z^n\setminus \reccone (P)$ such that $P + \linhull{v}$ is relatively lattice-free.
\end{quote}
\end{framed}

We call $\vRLF$ the generalization of $\RCGR$ where we do not ask for the input polyhedron to be integral. We show the following.

\begin{thm}\label{thm:algo-positive}
$\RCGR$ is decidable. Moreover, it can be decided in polynomial time if the dimension $n$ is fixed.
\end{thm}

\begin{thm}\label{thm:algo-negative}
$\vRLF$ is $\conp$-hard.
\end{thm}

Let us discuss some consequences of Theorems \ref{thm:algo-positive} and \ref{thm:algo-negative}. First, recall that an ${\cal H}$-polyhedron can be transformed into a ${\cal V}$-polyhedron in polynomial time in fixed dimension, and vice-versa. So, Theorem \ref{thm:algo-positive} also holds if the input is a ${\cal V}$-polyhedron. Also, recall that a (widely believed) conjecture states that no $\conp$-hard problem lies in $\np$. This however does not completely settle the complexity of $\RCGR$, as $\vRLF$ is a more general problem than $\RCGR$. It is not clear however how knowing that the input polyhedron is integral could help: recall that, for instance, it is unlikely that there exists a compact certificate for the integrality of a polyhedron.

In the rest of the section, we prove Theorems \ref{thm:algo-positive} and \ref{thm:algo-negative}.

\subsection{Proof of Theorem \ref{thm:algo-negative}}\label{sec:algo-aux}

Consider the following problem.

\begin{framed}
(\IPRI)
\begin{quote}%(\IPRI)

  Given: a polyhedron $P=\{x \in \R^n: Ax \leq b\}$, where $A \in \Z^{m \times n}$ and $b \in \Z^m$;

  Decide: if $P$ has an integer point in its relative interior.
\end{quote}
\end{framed}

Let $\IF$ be the problem of deciding if a polyhedron (given as a finite list of rational inequalities) contains an integer feasible point. Recall that $\IF$ is $\np$-complete, and can be solved in polynomial time in fixed dimension. 

\begin{lem}\label{lem:np-complete}
$\IPRI$ is $\np$-complete, and can be solved in polynomial time in fixed dimension.
\end{lem}

\begin{proof} One immediately reduces $\IF$ to $\IPRI$. Let in fact $P=\{x \in \R^n : Ax \leq b\}$, with $A \in \Z^{m \times n}$ and $b \in \Z^m$, and $\overline b \in \R^m$ be obtained from $b$ by adding $1/2$ to all its components. Note that $\overline P=\{x \in \R^n: Ax \leq \overline b\}$ contains exactly the same integer points as $P$ and no integer points on its boundary. On the other hand, $\IPRI$ can be decided in polynomial time in fixed dimension as follows: from the input system, detect a minimum defining system ${\cal C}$ for $P$; for each inequality $ax \leq \beta$ from ${\cal C}$ defining a facet $F$ of $P$, check if $F$ contains an integer point. If yes, then replace $ax \leq \beta$ with $ax \leq \beta - \varepsilon$, with $\varepsilon>0$ small enough. Call $\overline P$ the new polyhedron. Then $P$ is a yes-instance for $\IPRI$ if and only if $\overline P$ is a yes-instance for $\IF$. 
\end{proof}

Given Lemma \ref{lem:np-complete}, the proof of Theorem \ref{thm:algo-negative} is now immediate, as a polyhedron $P\subseteq \R^n$ is a yes-instance to $\IPRI$ if and only if the polyhedron $P \times \{0\} \subseteq \R^{n+1}$ is a no-instance to $\vRLF$.

\subsection{Proof of Theorem \ref{thm:algo-positive}}

In order to devise a fine procedure for $\RCGR$, we delve into the geometric characterization from Theorem \ref{th:main}. Our first observation is that we can reduce to the bounded case when $P$ is relatively lattice-free.

\begin{lem}\label{lem:out-lin}
Let $P\subseteq \R^n$ be an integral relatively lattice-free polyhedron with lineality space $S$.
\begin{itemize}
\item[(a)] If $\reccone(P) \neq S$, then $r^*(P)=+ \infty$.
\item[(b)] Otherwise, let $P=P'+S$, where $P'$ is an integral relatively lattice-free polytope such that $\dime(P')+\dime(S)=\dime(P)$. Then $r^*(P)=+\infty$ if and only if $r^*(P')=+\infty$, where the latter is computed in the affine hull of $P'$.
\end{itemize}
\end{lem}

\begin{proof}
Suppose there exists $v \in \reccone(P)\setminus S$. Then $P+\linhull v$ is relatively lattice-free, proving $(a)$.
Suppose now $S=\reccone(P)\neq \{0\}$, else $(b)$ is trivial. Up to a unimodular transformation, we can assume that $S$ is generated by a subset of vectors $e_k,\dots, e_n$ for some integer $k\leq n$, and that $\affhull (P')$ is generated by $e_1,\dots, e_{k-1}$. In order to investigate if $r^*(P)=+\infty$, we can restrict to investigate $P+\linhull v$ for integral vectors $v$ lying in the affine hull of $P'$. So let $v \in \affhull(P')$. Let $x' \in \Z^{k-1}$ and $x'' \in \Z^{n-k+1}$. As ${x' \choose x''} \in P + \linhull v$ if and only if ${x' \choose 0} \in P' + \linhull v$, (b) follows.
\end{proof}

\noindent Combining Theorem \ref{th:main} with Observation \ref{obs:lf}, one immediately deduces the following.

\begin{Rm}\label{rem:non-full-dim}
Let $P$ be an integral polyhedron that is not-full dimensional. Then $r^*(P)=+\infty$ if and only if $P$ is not relatively lattice-free.
\end{Rm}

For full-dimensional polyhedra, the situation is different. Clearly integral non-lattice-free polytopes have finite reverse CG rank. The following statement shows that lattice-free integral polytopes with finite reverse CG rank are somehow under control.

%The following Lemma exploits results from \cite{nizi}. The connection between our results and \cite{nizi} was pointed out to us by Gennadiy Averkov.

\begin{lem}\label{lem:averkov}
For each $n \in \N$, there are, up to unimodular transformations, only a finite number of $n$-dimensional lattice-free integral polytopes with $r^*<+\infty$. Each of these polytopes has volume at most $c^{4^n}$, for some constant $c>1$.
\end{lem}

\begin{proof}
\cite[Theorem 2.1]{nizi} implies the following: for each $n \in \N$, up to unimodular transformations, there exists only a finite number of $n$-dimensional lattice-free integral polytopes of $\R^n$ such that $P+\linhull v$ is not lattice-free for every $v \in \Z^n\setminus\{0\}$. The volume of each of those exceptions is at most $c^{4^n}$, for some constant $c>1$. By applying Theorem \ref{th:main}, the claimed result immediately follows. \end{proof}

We now prove Theorem \ref{thm:algo-positive}. Let $P$ be an input of $\RCGR$. Note that we can check if we are in case $(a)$ or $(b)$ of Lemma \ref{lem:out-lin} and, in case $(b)$, obtain $P'$ in time polynomial in $n$. So we assume wlog that $P$ is a polytope. Also, we can check if $P$ is relatively lattice-free in polynomial time in fixed dimension using Lemma \ref{lem:np-complete}. If it is not, then we output ``no'' because of Observation \ref{rem:non-full-dim} and the subsequent discussion. If conversely it is relatively lattice-free and not full-dimensional, then we output ``yes'' (see again Observation \ref{rem:non-full-dim}). Hence we can also assume $P$ full-dimensional and lattice-free.

We first show that $\RCGR$ is decidable. By Proposition \ref{prop:full-dim-nec}, we just need to give the following two finite procedures, which are then executed in turns (i.e. we alternate one step of the first and one step of the other) until one of the two halts.
\begin{itemize}
\item[INF]
If $r^*(P) = + \infty$,
this procedure finds a vector $v \in \Z^n$ such that $P + \linhull v$  is lattice-free.
\item[FIN]
If $r^*(P)$ is finite,
this procedure finds $k \in \N$ such that all relaxations of $P$ are contained in $P_k$.
\end{itemize}

The procedure INF enumerates all the possible vectors $v \in \Z^n$ by increasing norm.
For each candidate $v$, INF constructs an integer matrix $C$ and an integer vector $d$ such that $P + \linhull v = \{x \in \R^n : Cx\le d\}$, and checks if $P + \linhull v$ is lattice-free using Lemma \ref{lem:np-complete}.

\smallskip

The procedure FIN checks if all relaxations of $P$ are contained in $P_k$, for a fixed $k \in \N$. If so it stops, and if not, it checks $P_{k+1}$, and so on.
We now explain how FIN checks if all relaxations of $P$ are contained in $P_k$ in finite time.

Let $F$ be a facet of $P_k$ defined by inequality $cx \le \delta$, and let $H^F = \{x : cx = \delta\}$.
For every integer point $x^i \in P_k \setminus P$ (they are a finite number), there exists a polytope $R^F_i \subseteq H^F$ such that, for every point $r \in H^F$, we have $x^i \in \conv(r,P)$ if and only if $r \in R^F_i$.
%Note that, if $R^F_i \neq \emptyset$, then it has dimension $n-1$.
The polytope $R^F_i$ can be obtained by intersecting the hyperplane $H^F$ with the translated cone $C^F_i= x^i - \cone (W)$, where $W$ is the set of vectors $w$ such that $x^i + w$ is a vertex of $P$. It is clearly rational.

It can be checked that all relaxations of $P$ are contained in $P_k$ if and only if, for every facet $F$ of $P_k$, we have $\cup_{i \in I} R^F_i \supseteq F$, where $I = \{i : x^i \in \Z^n \cap P_k \setminus P\}$.
%For every facet $F$, by working in the affine space $H^F$, the problem reduces to the following: Given polytopes $F, R_i$, is it true that $\cup_i R_i \supseteq F$?
Hence to conclude the procedure FIN, it is sufficient to show a finite algorithm solving the following problem: given rational polyhedra $F,\{R_i\}_{i \in I}\subseteq \R^n$, with $|I| \in \N$, decide if $\cup_{i \in I} R_i \supseteq F$. It can be tested in finite time by induction on $|I| + \dim (F)$. We can assume that $F$ is full-dimensional, as otherwise we can work in the affine space $\aff (F)$ by intersecting all polyhedra $R_i$ with $\aff (F)$.
As polyhedra are closed sets, we can also assume that all polyhedra $R_i$ are full-dimensional, because we can always ignore those with dimension strictly smaller than the dimension of $F$.
The base cases are when either $\dim (F) =0$, or when $|I|=1$.
The first case is trivially solvable, while the second can be solved by linear programming.
For the inductive step, let $\bar \imath \in I$.
Clearly $\cup_{i \in I} R_i \supseteq F$ if and only if for every facet $ax \le \beta$ of $R_{\bar \imath}$, we have
\begin{equation} \label{whatever}
\cup_{i \in I}  R_i \supseteq \overline F,
\end{equation}
where $\overline F = \{x \in F : ax \ge \beta\}$.
If the polyhedron $\overline F$ is not full-dimensional, then the problem can be solved by induction.
Otherwise, if $\overline F$ is full-dimensional, \eqref{whatever} happens if and only if $\cup_{i \in I \setminus \{\bar \imath\}}  R_i \supseteq \bar F$, which can also be solved by induction.

\medskip

\noindent We now prove that $\RCGR$ can be solved in polynomial time in fixed dimension. Let ${\cal C}$ be the class of equivalence (under unimodular transformations) of lattice-free integral full-dimensional polytopes of $\R^n$ with finite reverse CG rank. From Lemma \ref{lem:averkov} we know that ${\cal C}$ is finite and that all polytopes belonging to some class from ${\cal C}$ have volume at most $c^{4^n}$. It is proved in \cite[Theorem 2]{Lazi} that an integer $n$-dimensional polytope of volume at most $V$ can be mapped via a unimodular transformation to an integral polytope contained in the cube of side at most $V \cdot n \cdot n! $. Hence, each class from ${\cal C}$ has a representative that is contained in the cube $K_n\subseteq \R^n$ of side $c^{4^n} \cdot n \cdot n!$. Because of the first part of the proof, we can construct this family ${\cal R}$ of representatives (with repetitions allowed) in fixed time for fixed dimension. To conclude, we only need to check if $P$ can be mapped via a unimodular transformation to some $\overline P \in {\cal R}$. This can be done in polynomial time as follows. Fix $\overline P \in {\cal R}$. Observe that, if $P$ and $\overline P$ are equivalent up to a unimodular transformation, they satisfy the following conditions: they have the same number of vertices, say $v_1,\dots, v_t$ for $P$ and $w_1,\dots, w_t$ for $\overline P$. Moreover, $\vol(P)=\vol(\overline P)$ must hold, so all affine transformations $A (\cdot) + u$ mapping $P$ to $\overline P$ are such that $|\det(A)|=1$. There exists a unimodular transformation mapping $P$ to $\overline P$ if and only if any solution $(A^*,u^*)$ to the following instance of $\IF$ with variables $(A,u)$:
$$\begin{array}{llll}
Av_j + u & = & w_j & \hbox{ for $j=1,\dots, t$} \\
A & \in & \Z^{n \times n}\\
u & \in & \Z^n,
\end{array}
$$

\noindent is such that $|\det(A^*)|=1$. The statement then follows from the fact that the determinant of $A^*$ and the vertices of $P$ can be computed, and $\IF$ solved, in polynomial time in fixed dimension.

%for the special case of polytopes. We are left to show it in the unbounded case, so let $P\subseteq \R^n$ be an $n$-dimensional unbounded lattice-free polyhedron with $r^*(P)<+\infty$. Note that the recession cone and the lineality space of $P$ coincide, else  $P+\linhull v$ is lattice-free for each $v \notin \reccone(P)$ such that $-v \in \reccone(P)$, contradicting Theorem \ref{th:main}. Then $P=P'+L$, where $P'$ is a relatively lattice-free integral polytope, $L$ is a linear subspace, and $\dime(P')+\dime(L)=n$. Up to a unimodular transformation, we can assume that $\affhull (P')=\{ x \in \R^n: x_{n'+1}=x_{n'+2}=\dots=x_{n}=0\}$ for some $n'<n$ and that $L$ is its orthogonal complement. As $P+\linhull v$ is not lattice-free for each $v \in \Z^n\setminus L$, then $P'+\linhull v$ is not lattice-free for every $v \in \aff(P') \setminus \{0\}$. From what proved for polytopes, up to a unimodular transformation $P'$ belongs to a finite family. The statement follows.

\section{Extensions}\label{sec:extens}

\subsection{On the definition of relaxation}

Recall that we defined a relaxation of an integral polyhedron to be a {\em rational} polyhedron. We remark that the rationality assumption is crucial in the statement of Theorem~\ref{th:main}. As an example, consider the polytope $P\subseteq \R^2$ consisting only of the origin. Any line $Q\subseteq \R^2$ passing through the origin and having irrational slope is an (irrational) polyhedron whose integer hull is $P$. One readily verifies that the CG closure of $Q$ is $Q$ itself, showing that in this case the CG closures of $Q$ do not converge to the integer hull $P$. However, no vector $v$ satisfying the conditions of Theorem~\ref{th:main} exists.

Assume now that $P$ is a polytope. As a referee pointed out, Theorem \ref{th:main} still holds if we define a relaxation of $P$ to be a convex body whose integer hull is $P$. This immediately follows from the fact that, for each convex body, one can construct a rational polytope that contains it and has the same integer hull (see e.g.~\cite[Proof of Corollary 23.4a]{sch}). The example above shows that a further extension of the definition of relaxation as to include all unbounded convex sets will make Theorem \ref{th:main} false. In particular, it is not clear if we can substantially extend the concept of relaxation for unbounded integral polyhedra $P$ and keep Theorem \ref{th:main} true.

\subsection{On relaxations with bounded facet complexity}

An interesting question is whether bounding the facet complexity of relaxations leads to stronger bounds on the CG rank of polyhedra (we refer the reader to~\cite{sch} for the concept of \emph{facet complexity}, as well as for the one of \emph{vertex complexity}). As the example from Figure \ref{fig:2-dim} shows, the CG rank of a polyhedron can be exponential in the facet complexity of the polyhedron itself. But what happens to the reverse CG rank of an integral polyhedron $P$, if we require all relaxations to be of facet complexity $\nu(c)$ for some fixed function $\nu$, where $c$ is the facet complexity of $P$? Recall the following two results (see respectively e.g.~\cite[Theorem 10.2]{sch} and~\cite[Theorem 16.1]{sch}): \emph{(i)} there exists a polynomial $\phi: \N \times \N \rightarrow \N$ such that, if $P\subseteq \R^n$ has facet complexity at most $c$, then $P$ has vertex complexity at most $\phi(n,c)$; \emph{(ii)} for each integral polyhedron $P$ and each relaxation $Q$ of $P$, $\reccone(P)=\reccone(Q)$. From \emph{(ii)}, it follows that each relaxation $Q$ of a (non-empty, integral) polyhedron $P\subseteq \R^n$ can be written as $\conv(V\cup S) + \reccone(P)$, where $V$ is composed of a point from each minimal non-empty face of $P$ and $S$ is a finite set of points. Then from \emph{(i)} we have that bounding the complexity of the inequalities defining $Q$ implies bounding the complexity of each point in $S$, and hence its $\ell_\infty$ norm. Using the notation from the proof of Theorem \ref{th:main}, we deduce that $V \cup S \subseteq P_k$, for some $k$ big enough that does not depend on $Q$ but only on $P$ and $\nu$, and consequently that $Q\subseteq P_k$, since $\reccone(Q)=\reccone(P)=\reccone(P_k)$. Repeating the arguments from the proof of the implication $(1) \Rightarrow (2)$ in Proposition \ref{prop:full-dim-nec}, we deduce that the CG rank of $Q$ is bounded by a function depending only on $P$, $\nu(c)$, and $k$. Hence, the reverse CG rank in this case would be always bounded.

Let us remark that this is in sharp contrast with what happens in the mixed-integer case, where a relaxation of small complexity may have unbounded CG (or even split) rank: see~\cite[Example 2]{CKS}.

\subsection{Reverse split rank} Another interesting problem is the extension of the concept of reverse CG rank to the case of split inequalities. It can be proved that in dimension 2 the split rank of every rational polyhedron is at most 2. That is, the \emph{reverse split rank} (defined in the obvious way) is bounded by a constant in dimension 2, while recall that this is not true for the reverse CG rank, see Section \ref{sec:intro}. However, one can prove that already in dimension 3 a constant bound does not exists, as implied by Lemma \ref{lem:split}. We remark that a characterization of integral polyhedra with finite reverse split rank is given in~\cite{split}. However,~\cite{split} postdated the current paper, and in fact builds on it.

\begin{lem}\label{lem:split}
Let $T\subseteq \R^3$ be the triangle $\conv\{(0,0,0),(0,2,0),(2,0,0)\}$. The reverse split rank of $T$ is $+\infty$.
\end{lem}
\begin{proof}
For $h \in \Q_+$, let {$x^h$}$=(\frac{1}{2},\frac{1}{2},h)\in \R^3$ and $T^h=\conv(T;x^h)$. Note that, for each $h \in \Q_+$, $(T^h)_I=T$.
We assume that the reader is familiar with basic definitions on split cuts, which can be found e.g. in~\cite{CKS}.

For a polyhedron $P$, we denote by $SC(P)$ its split closure and by $SC^k(P)$ its $k-$th split closure. For an inequality $\alpha x \leq \beta$, we let $SC_{\alpha,\beta}^{\leq}(P)$ be the set of points in $P$ that satisfy $\alpha x \leq \beta$, and $SC_{\alpha,\beta}^{\geq}(P)$ the set of points in $P$ that satisfy $\alpha x \geq \beta+1$. Also, we let $SC_{\alpha,\beta}(P)=\conv\{SC_{\alpha,\beta}^{\leq}(P),SC_{\alpha,\beta}^{\geq}(P)\}$. We say that a split $S=\{x : \beta < \alpha x < \beta +1\}$ \emph{cuts off} a point $\bar x$ if $\bar x \notin SC_{\alpha,\beta}(P)$. Note that if $S$ cuts off a point $\bar x$, necessarily $\beta < \alpha \bar x < \beta + 1$.

We prove that, for $h \geq 4/3$, there exists no split $\{x:\beta < \alpha x < \beta +1\}$ that cuts off both $x^h$ and $x^{h/4}$ from $T^h$, i.e. such that $\{x^h, x^{h/4}\} \cap SC_{\alpha, \beta} (T^h) = \emptyset$. This implies that
for $h \geq \frac{4}{3}$, $x^{h/4}\in SC(T^h)$. Thus, $SC(T^h)\supseteq T^{h/4}$ and consequently, $SC^k(T)\supseteq T^{h/4^k}$. Hence, $\Omega(\log(h))$ rounds of the split closure are needed to obtain the integer hull starting from $T^h$, and the statement follows.

Fix a rational $h\geq 4/3$, and a split $S=\{x: \beta < \alpha x < \beta +1\}$, with $\alpha=(\alpha_1,\alpha_2,\alpha_3) \in \Z^3$, $\beta \in \Z$. We first deal with splits where $\alpha_3=0$. Note that $\bar x=(1,\frac{1}{2},\frac{h}{2})\in T^h$ (resp. $\tilde x=(\frac{1}{2},1,\frac{h}{2}) \in T^h$), since this point is in the segment between $(\frac{3}{2},\frac{1}{2},0)$ and $x^h$ (resp. between ($\frac{1}{2},\frac{3}{2},0)$ and $x^h$). If $S$ cuts off $x^{h/4}$, then $\bar x, \tilde x \in S$, since $x^{h/4}$ is in the segment between $\bar x$ and $(0,\frac{1}{2},0)$ (resp. $\tilde x$ and $(\frac{1}{2},0,0)$) and the latter belongs to $T$ (hence also to $SC(T^h)$). Moreover, the split also needs to cut off $x^{h/4}$, hence $x^{h/4}\in S$. Thus, $\alpha_1,\alpha_2,\beta$ satisfy
$$ \beta < \alpha_1 + \frac{1}{2} \alpha_2, \frac{1}{2} \alpha_1 + \alpha_2,  \frac{1}{2} \alpha_1 + \frac{1}{2}\alpha_2 < \beta +1,$$
\noindent which, since $\alpha_1$, $\alpha_2$, and $\beta$ can be assumed to be integer, imply
$$ 2 \alpha_1 + \alpha_2 = \alpha_1 + 2 \alpha_2 = \alpha_1 + \alpha_2 = 2\beta +1.$$
\noindent The unique solution to the system above is $\alpha_1=0$, $\alpha_2=0$, $\beta=-\frac{1}{2}$, which is not an integral vector.

Thus, we can assume $\alpha_3 \neq 0$. As the split needs to cut off both $x^h$ and $x^{h/4}$, we have
$$ \beta < \frac{1}{2}\alpha_1 + \frac{1}{2} \alpha_2 + h \cdot \alpha_3,  \frac{1}{2}\alpha_1 + \frac{1}{2} \alpha_2 + \frac{h}{4}  \cdot \alpha_3 < \beta +1.$$
\noindent Setting $\overline \beta=\beta - \frac{1}{2}\alpha_1 - \frac{1}{2} \alpha_2$, it follows $ \overline \beta < h \cdot \alpha_3,  \frac{h}{4} \cdot \alpha_3 < \overline \beta +1$. Hence, using the fact that $|\alpha_3|\geq 1$, we obtain $1>|h\alpha_3 - \frac{h}{4}\alpha_3| = \frac{3}{4}h|\alpha_3|\geq \frac{3}{4}h\geq 1$, a contradiction.
\end{proof}

%Note $P'+\linhull v$ is relatively lattice-free for some $v \in \Z^{n+1}$ if and only if $P$ is relatively lattice-free. This reduces $\IPRI$ to $\vRLF$, showing that $\vRLF$ is $\np$-hard. Now let $P$ be a no-instance to $\IPRI$, that is, $P$ contains no integer point in its relative interior. If $\IPRI \in \np$, then there is a polynomial-size certificate showing that $P'$ is relatively lattice-free. This is also a polynomial-size certificate showing that $P$ is relatively lattice-free, hence implying that $\IPRI$ is in $\conp$.\hspace*{\stretch{1}}

%\hspace*{\stretch{1}}

\subsection*{Acknowledgments}

We thank Gennadiy Averkov for pointing out the connection between Theorem \ref{th:main} and \cite{nizi}, that lead to Lemma \ref{lem:averkov}.

\bibliographystyle{plain.bst}

\end{document}